\documentclass{article}
\usepackage{graphicx} 

\usepackage[english]{babel}
\usepackage[T1]{fontenc}

\usepackage{tikz}
\tikzstyle{vertex}=[circle, draw, inner sep=2pt, minimum size=6pt]

\usepackage{mathrsfs}

\usepackage[a4paper,top=3cm,bottom=2cm,left=3cm,right=3cm,marginparwidth=1.75cm]{geometry}

\usepackage{comment}
\usepackage{amsfonts}
\usepackage{fullpage}
\usepackage{latexsym}
\usepackage{amsfonts}
\usepackage{blkarray}
\usepackage{graphicx}
\usepackage{float}
\usepackage{enumerate}
\usepackage{amsthm,amsmath,amssymb}
\usepackage{verbatim}
\usepackage{enumitem}
\usepackage{blkarray}
\usepackage{mathtools}
\usepackage{algorithm}
\usepackage[noend]{algpseudocode}

\providecommand{\keywords}[1]{
  \small	
  \textbf{\textit{Keywords---}} #1
}

\def\T{^\top}
\newtheorem{theorem}{Theorem}[section]
\newtheorem{lemma}{Lemma}[section]
\newtheorem{corollary}{Corollary}[section]
\newtheorem{proposition}{Proposition}[section]

\newtheorem{example}{Example}[section]

\newtheorem{observation}{Observation}[section]
\newcommand{\nc}{\newcommand}
\newcommand{\Yildirim}{Y{\i}ld{\i}r{\i}m}

\nc{\cA}{{\cal A}}
\nc{\cB}{{\cal B}}
\nc{\cC}{{\cal C}}
\nc{\cD}{{\cal D}}
\nc{\cE}{{\cal E}}
\nc{\cG}{{\cal G}}
\nc{\cF}{{\cal F}}
\nc{\cH}{{\cal H}}
\nc{\cI}{{\cal I}}
\nc{\cK}{{\cal K}}
\nc{\cL}{{\cal L}}
\nc{\cM}{{\cal M}}
\nc{\cN}{{\cal N}}
\nc{\cO}{{\cal O}}
\nc{\cP}{{\cal P}}
\nc{\cQ}{{\cal Q}}
\nc{\cR}{{\cal R}}
\nc{\cS}{{\cal S}}
\nc{\cT}{{\cal T}}
\nc{\cV}{{\cal V}}
\nc{\tx}{{\tilde x}}
\nc{\la}{{\langle}}
\nc{\ra}{{\rangle}}
\nc{\ts}{\textsuperscript}
\def\R{\mathbb{R}}

\date{\today}

\title{On Tractable Convex Relaxations of Standard Quadratic Optimization Problems under Sparsity Constraints}

\author{Immanuel Bomze\thanks{VCOR and Research Network Data Science, University of Vienna, Oskar-Morgenstern-Platz 1,
1090 Wien, Austria. ORCID ID: 0000-0002-6288-9226 E-mail: {\tt immanuel.bomze@univie.ac.at}} \and Bo Peng\thanks{VGSCO and ISOR, University of Vienna, Oskar-Morgenstern-Platz 1,
1090 Wien, Austria. ORCID ID: 0000-0002-2650-0295 E-mail: {\tt bo.peng@univie.ac.at}} \and Yuzhou Qiu\thanks{School of Mathematics, Peter Guthrie Tait Road, The University of Edinburgh, Edinburgh, EH9 3FD, United Kingdom. E-mail: \tt{y.qiu-16@sms.ed.ac.uk}} \and E. Alper Y{\i}ld{\i}r{\i}m\thanks{School of Mathematics, Peter Guthrie Tait Road, The University of Edinburgh, Edinburgh, EH9 3FD, United Kingdom. ORCID ID: 0000-0003-4141-3189 E-mail: \tt{E.A.Yildirim@ed.ac.uk}}}
\date{September 28, 2023}

\begin{document}

\maketitle

\begin{abstract}
Standard quadratic optimization problems (StQPs) provide a versatile modelling tool in various applications. In this paper, we consider StQPs with a hard sparsity constraint, referred to as sparse StQPs. We focus on various tractable convex relaxations of sparse StQPs arising from a mixed-binary quadratic formulation, namely, the linear optimization relaxation given by the reformulation-linearization technique, the Shor relaxation, and the relaxation resulting from their combination. We establish several structural properties of these relaxations in relation to the corresponding relaxations of StQPs without any sparsity constraints, and pay particular attention to the rank-one feasible solutions retained by these relaxations. We then utilize these relations to establish several results about the quality of the lower bounds arising from different relaxations. We also present several conditions that ensure the exactness of each relaxation.
\end{abstract}



\section{Introduction}

The Standard Quadratic optimization Problem (StQP) consists of minimizing a quadratic form over the standard simplex (all vectors with no negative coordinates that sum up to one).

Since no assumptions on the definiteness of quadratic form are made, this problem class is NP-hard. Indeed, the maximum-clique problem can be reduced to~\eqref{stqp}~\cite{motzkin_straus_1965}. Therefore, we can view the class of StQP as the simplest of the hard problems: the simplest non-convex objective functions are generated by indefinite Hessians, and  the feasible set is the simplest bounded polyhedron (polytope) with a very obvious structure of faces comprised of standard simplices in lower dimensions when some variables are fixed to zero.

Despite its simplicity, the class of StQPs provides a quite versatile modelling tool (see, e.g.,~\cite{Bomze98}). Applications are numerous, ranging from the famous Markowitz portfolio problem in finance, evolutionary game theory in economics and quadratic resource allocation problems, through machine learning (background–foreground clustering in image analysis), to the life sciences --— e.g., in population genetics (selection models) and ecology (replicator dynamics). 

StQPs appear also quite naturally as subproblems in copositive-conic relaxations of mixed-integer or combinatorial optimization problems of all sorts. Finally, using barycentric coordinates, every quadratic optimization problem over a polytope with known (and not too many) vertices can be rephrased as an StQP.

The aforementioned structural simplicity does not preclude coexistence of an exponential number of (local or global) solutions to some StQPs. Some of these solutions may be sparse (and will be so with high probability in the average case, see below), others may have many positive coordinates. However, in important applications like some variants of sparse portfolio optimization problems where one is interested in investments with a limited number of assets (see, e.g.,~\cite{MDA2019} and the references therein), sparsity of a solution must be enforced by an additional, explicit hard constraint on the number of positive coordinates. Introducing this sparsity constraint can render StQPs NP-hard even if the Hessian is positive-definite. 

This paper deals with such problems and investigates the structural properties of tractable linear and semidefinite relaxations which scale well with the dimension.

\section{Background, Motivation, and Layout of Contribution}

In this section, we provide some background on standard quadratic optimization problems. We present our motivation for studying the variant with a hard sparsity constraint. We introduce our notation and give an outline of the paper.

\subsection{The Combinatorial Nature of Standard Quadratic Optimization -- Coexistence of Solutions and Role of Active Sets}

The well-studied Standard Quadratic optimization Problem (StQP) is given by
\begin{equation}\tag{StQP}\label{stqp}
\ell(Q) := \min\limits_{x \in \R^n} \left\{x\T Q x: x \in F\right\},
\end{equation}
where $Q \in \cS^n$ is the problem data, $x \in \R^n$ is the decision variable, and $F \subset \R^n$ denotes the standard simplex given by
\begin{equation} \label{def_F}
F := \left\{x \in \R^n: e\T  x = 1, \quad x \geq 0\right\},
\end{equation}
where $e \in \R^n$ denotes the vector of all ones. 

There is an exponential number, namely $2^n-1$, of faces of $F$, which form the ``combinatorial'' reason for NP-hardness. 
Indeed, if the active set $\{ i : x_i^* =0 \}$ at the global solution $x^*$ is known exactly, locating the solution (i.e., determining $x^*$ or a value-equivalent alternative with the same set of zero coordinates) reduces to solving an $n\times n$ linear equation system. The same holds true for locating local solutions and even first-order critical (KKT) points. 
This phenomenon may be the reason why recently iterative first-order methods were proposed, which can achieve identification of the correct active set in finite time~\cite{Bomz19a}. 

For any instance of \eqref{stqp}, not all faces of $F$ can contain an isolated (local or global) solution in their relative interior, as there is an upper bound on their cardinality given by Sperner's theorem on the maximal antichain (and Stirling's asymptotics), namely
\begin{equation}\label{sperstir}    
 \binom{n}{\lfloor \frac n2\rfloor} \sim \sqrt{\frac 2{\pi n}}\, 2^n\quad\mbox{as } n\to \infty\, .\end{equation}
Scozzari and Tardella~\cite{Scoz08} show that solutions can occur only in the relative interior of a face restricted to which the objective function is strictly convex. Nevertheless,
recent research~\cite{Bomz16a} has shown an exponential behavior regarding the number of local (or global) solutions: in the worst case, an instance of \eqref{stqp} of order $n$ can have at least
\begin{equation}\label{lowbd}
 (15120)^{n/4} \approx (1.4933)^n   
\end{equation}
coexisting optimal solutions, a lower bound that currently seems to be the largest one known. The other bad news is that rounding on the standard simplex is, from the asymptotic point of view, also not always successful~\cite{Bomz14c}. In spite of all this, \eqref{stqp} admits a polynomial-time approximation scheme (PTAS)~\cite{Bomz02c}.

\subsection{Worst-case versus Average-Case Behavior -- Expected Sparsity}

All of the above observations refer to the worst case, of course. Several researchers turned to the average case, modelled by randomly chosen instances. Already in 1988, Kingman~\cite{King61a} observed that very large polymorphisms (i.e., solutions $x^*$ with more than $C\sqrt n$ positive coordinates) are atypical. More recently, in a series of papers Kontogiannis and Spirakis~\cite{Kont05,Kont09,Kont10} looked at models with several independent and identically distributed (e.g., Gaussian or uniform) entries of $Q \in \cS^n$
and proved, among other results, that the expected number of (local) solutions does not grow faster than $\exp(0.138n)\approx (1.148)^n$, way smaller than the worst-case lower bound in~\eqref{lowbd}. Based upon more recent research by Chen and coauthors~\cite{chen2015new,chen2013sparse}, under quite reasonable distributional assumptions modeling the random average case, the probability that the global solution has more than 2 positive coordinates (i.e., that it does not lie on an edge of $F$) is asymptotically vanishing faster than
$$K \frac {(\log n)^2}n\quad\mbox{with } n\to \infty\, ,$$
where $K>0$ is a universal constant~\cite[Proposition~1]{Bomz16a}.

\subsection{StQPs with a Hard Sparsity Constraint}

However, if the instances are somehow structured, we cannot rely on our ``luck'' that $Q$ exhibits an average behavior in the above sense, and still, we may prefer a sparse solution to~(\ref{stqp}). So, in pursuit of these sparse solutions, we introduce the following variant under a cardinality constraint, referred to as the \emph{sparse StQP}: 
\begin{equation*}
 \ell_\rho (Q) := \min\limits_{x \in \R^n} \left\{x\T  Q x: x \in F_\rho\right\}, 
\end{equation*}
where 
\begin{equation} \label{def_F_rho}
F_\rho := \left\{x \in F : \quad \|x\|_0 \leq \rho\right\}\, .
\end{equation}
Here, $\|x\|_0$ denotes the number of nonzero components of a vector $x$ and $\rho \in \{1,\ldots,n\}$ is the sparsity parameter. 

The elements of $F_\rho$ will be referred to as $\rho$-sparse.  When $\rho$ is fixed independently of $n$, $F_\rho$ is the union of  ${\mathfrak O} (n^\rho)$ faces of $F$, a number  polynomial in $n$. In each of these faces, due to~\eqref{sperstir}, at most
$\binom{\rho}{\lfloor \frac \rho 2\rfloor}$ local solutions to~\eqref{def_F_rho} can coexist, so we end up with a polynomial set of candidates which makes problem~\eqref{def_F_rho} solvable in polynomial time, again for universally fixed $\rho$. However, if $\rho$ may increase with $n$, e.g. $\rho=  \lfloor \frac n2\rfloor  $, or even $\rho =n$, the above observations show that the sparse StQP is NP-hard even when $Q$ is positive semidefinite due to the combinatorial nature of the sparsity term $\|x\|_0$.

Evidently, any (feasible or optimal) solution of the sparse StQP is a feasible solution to~\eqref{stqp} with guaranteed $\rho$-sparsity, which can be crucial. Even if $\rho $ is fixed to a moderate number, say to 6, and for medium-scale dimensions, say $n=100$, polynomial worst-case behavior would not help much in practical optimization since $n^\rho=10^{12}$. This emphasizes the need for tractable relaxations of the sparse StQP. 

We start with some simple observations.

\begin{lemma} \label{simple-obs1}
The following relations hold:
\begin{equation} \label{rels1}
 \ell(Q) = \ell_n(Q) \leq \ell_{n-1}(Q) \leq \ldots \leq \ell_2(Q)\leq\ell_1(Q) \, ,
\end{equation}
with \begin{equation} \label{rho1}
\ell_1(Q) = \min\limits_{1\le k \le n} Q_{kk} \, ,\end{equation} 
and
\begin{equation} \label{rho2}
\ell_2(Q) = \min\left\{\min \left\{ \textstyle{\frac{Q_{ii}Q_{jj}-Q_{ij}^2}{Q_{ii}+Q_{jj}-2Q_{ij}}}: Q_{ij} < \min\{Q_{ii},Q_{jj}\}, {1\le i<j\le n} \right \}, \, \ell_1(Q)\right\}\, .\end{equation}
Furthermore, we have $\ell(Q) = \ell_\rho(Q)$ if and only if (\ref{stqp}) has a $\rho$-sparse optimal solution.
\end{lemma}
\begin{proof}
The relations~\eqref{rels1} and \eqref{rho1} follow from $F_1 = \{e^1,e^2,\ldots,e^n\} \subset F_2 \subset \cdots \subset F_{n-1} \subset F_n = F$, where $F_\rho$ and $F$ are given by \eqref{def_F} and \eqref{def_F_rho}, respectively. For $\rho=2$, a straightforward discussion of univariate quadratics over the edges $\mathrm{conv} \left(\left\{e^i,e^j\right\}\right),~1 \leq i < j \leq n$ (in case these are strictly convex functions yielding a minimizer in the relative interior of the edge) is sufficient to establish~\eqref{rho2}. The last assertion is trivial.
\end{proof}

The condition $\ell(Q)=\ell_2(Q)$ is related to edge-convexity of the instance of \eqref{stqp} as discussed in~\cite[Theorem~1]{Scoz08} but we will not dive into details here. Rather observe that the effort to calculate $\ell_2(Q)$, obviously an upper bound of $\ell(Q)$, is the same as for the closed-form lower bound $\ell^{\rm ref}(Q)\leq \ell(Q)$ proposed in~\cite{Bomz08b}. The bracket 
$$\ell^{\rm ref}(Q)\leq \ell(Q) \leq \ell_2(Q)$$
shrinks to a singleton (i.e. the discussed bounds are exact) if and only if all off-diagonal entries of $Q$ are equal, in which case, an optimal solution $x^*$ to~\eqref{stqp} must satisfy $\|x^*\|_0 \leq 2$ (see~\cite[Theorem~2]{Bomz08b} and~\eqref{rho2}).

\subsection{Mixed-Binary Quadratic Formulation of Sparse StQPs and Contributions}

By introducing binary variables, the sparse StQP can be reformulated as a mixed-binary QP:
\[
\begin{array}{llrcl}
\tag{StQP($\rho$)}\label{sstqp} & \ell_\rho (Q) =\min\limits_{x \in \R^n} & x\T  Q x & & \\
 & \textrm{s.t.} & & & \\
 & & e\T  x & = & 1 \\ 
 & & e\T  u & = & \rho \\ 
 & & x & \leq & u \\
 & & u & \in & \{0,1\}^n  \\
 & & x & \geq & 0.
    \end{array}
\]

In this paper, we focus on various convex relaxations of \eqref{sstqp}, all more tractable than the conic ones presented in~\cite[Section~3.2]{Bomz23a} for general quadratic optimization problems. In particular, we establish several structural properties of these relaxations and shed light on the relations between each relaxation of \eqref{sstqp} and the corresponding relaxation of \eqref{stqp}. We then draw several conclusions about the relations between different relaxations as well as the strength of each relaxation.

While it turns out that all relaxations behave as expected for the case of $\rho = n$, already for the cases $\rho=1$ and $\rho=2$ (which cannot be excluded with a high probability in the random average case models) and other moderate sparsity values, there is a sharp contrast between the relaxations, which contributes to the motivation of this study. Typically, applications would require models with sparsity (significantly) less than half of the dimension, for which we obtain more interesting results. 

We will also pay particular attention to the case of rank-one solutions to the relaxations (all of them use matrix variables by lifting), in particular, because they certify optimality if optimal to the relaxed problems, and also because in algorithmic frameworks, we may (warm-)start with some (good) feasible solutions to the original problem of larger sparsity than desired.

\subsection{Notation and Organisation of the Paper}

We use $\R^n$, $\R^n_+$, $\R^{m \times n}$, and $\cS^n$ to denote the $n$-dimensional Euclidean space, the nonnegative orthant, the set of $m \times n$ real matrices, and the space of $n \times n$ real symmetric matrices, respectively.  We use 0 to denote the real number 0, the vector of all zeroes, as well as the matrix of all zeroes, which should always be clear from the context. We denote by $e \in \R^n$ and $e^i \in \R^n,~1 \leq i \leq n$, the vector of all ones and the $i$th unit vector, respectively. All inequalities on vectors or matrices are understood to be applied componentwise. For $A \in \cS^n$ and $B \in \cS^n$, we use $A \succeq B$ to denote that $A - B$ is positive semidefinite. 
For $x \in \R^n$ and an index set $\mathbf{K} \subseteq \{1,\ldots,n\}$, we denote by $x_{\mathbf{K}} \in \mathbb{R}^{\lvert\mathbf{K}\rvert}$ the subvector of $x$ restricted to the indices in $\mathbf{K}$, where $\lvert\cdot\rvert$ denotes the cardinality of a finite set. 
For singleton index sets, we simply use $x_j$ and $A_{ij}$ to denote the components of $x \in \R^n$ and $A \in \R^{m \times n}$. For $B \in \R^{n \times n}$ and $b \in \R^n$, we denote by $\textrm{diag}(B) \in \R^n$ and $\textrm{Diag}(b) \in \cS^n$ the vector given by the diagonal entries of $B$ and the diagonal matrix whose diagonal entries are given by $b$, respectively. The convex hull of a set is denoted by $\textrm{conv}(\cdot)$. For any $u \in \R^n$ and $v \in \R^n$, $u\T v$ denotes the Euclidean inner product. Similarly, for any $U \in \R^{m \times n}$ and $V \in \R^{m \times n}$, the trace inner product is denoted by 
\[
\langle U, V \rangle = \textrm{trace}(U^T V) = \sum\limits_{i=1}^m \sum\limits_{j = 1}^n U_{ij} V_{ij}.
\]

The paper is organized as follows. In Section~\ref{Sec2}, we consider several convex relaxations of (\ref{sstqp}). Section~\ref{RLT-Relaxation} focuses on the RLT (reformulation-linearization technique) relaxation of \eqref{sstqp} and presents several results in comparison with the RLT relaxation of (\ref{stqp}). The Shor relaxation of \eqref{sstqp} is treated in Section~\ref{SDP-Relaxation} and compared with that of \eqref{stqp}. In Section~\ref{SDP-RLT-Relaxation}, we then study the convex relaxation of \eqref{sstqp} given by combining the RLT and Shor relaxations and compare it with that of \eqref{stqp}. We conclude the paper in Section~\ref{conc}.

\section{Convex Relaxations: RLT and Shor} \label{Sec2}

In this section, we consider several well-known convex relaxations of (\ref{sstqp}), which use LP (linear programming) and SDP (semidefinite programming) methods. We study their properties and establish relations between each relaxation of (\ref{sstqp}) and the corresponding relaxation of \eqref{stqp}.

\subsection{RLT Relaxation} \label{RLT-Relaxation}

In this section, we consider the RLT (reformulation-linearization technique) relaxation of \eqref{sstqp} and compare it with the RLT relaxation of (\ref{stqp}). 

RLT relaxations of optimization problems with a quadratic objective function and a mix of linear and quadratic constraints are obtained by a two-stage process (see, e.g., \cite{Sherali1999}). The first stage, referred to as reformulation, consists of generating (additional) valid quadratic constraints from linear constraints by multiplying each pair of linear inequality constraints as well as each linear equality constraint by each variable. In the second stage, referred to as linearization, all of the original and additional quadratic functions are linearized by replacing the quadratic terms $x_i x_j$ by a lifted variable $X_{ij},~1 \leq i \leq j \leq n$. Together with the original linear constraints, this gives rise to the RLT relaxation. 

We first start with the RLT relaxation of (\ref{stqp}):
\[
\tag{R1}\label{R1} \quad \ell^{R1}(Q) := \min\limits_{x \in \R^n, X \in \cS^n} \left\{\langle Q, X \rangle: (x,X) \in \cF^{R1}\right\},
\]
where 
\begin{equation} \label{def_F_R1}
\cF^{R1} := \left\{(x,X) \in \R^n \times \cS^n: e\T  x = 1, \quad X e = x, \quad x \geq 0, \quad X \geq 0\right\}.   
\end{equation}

Note that $x \geq 0$ is a redundant constraint in $\cF^{R1}$ since it is implied by $X e = x$ and $X \geq 0$. Furthermore, it is easy to see that $\cF^{R1}$ is a polytope. We first recall the following result about \eqref{R1}.

\begin{proposition}[Qiu and Y{\i}ld{\i}r{\i}m~(2023)~\cite{QiuY23a}] 
\label{rlt-closed-form}
The set of vertices of $\cF^{R1}$ is given by 
\begin{equation} \label{rlt-vertices}
\left\{(e^i, e^i (e^i)\T ): i = 1,\ldots,n\right\} \cup \left\{\left(\frac{1}{2}(e^i+e^j), \frac{1}{2}(e^i (e^j)\T  + e^j (e^i)\T )\right): 1 \leq i < j \leq n\right\}.
\end{equation}
Therefore,
\[
\ell^{R1} (Q) = \min\limits_{1\leq i \leq j \leq n} Q_{ij} \leq \ell(Q).
\]
Furthermore, \eqref{R1} is exact (i.e., $\ell^{R1} (Q) = \ell(Q)$) if and only if 
\[
\min\limits_{1\leq i \leq j \leq n} Q_{ij} = \min\limits_{1\le k \le n} Q_{kk}.
\]
\end{proposition}

Proposition~\ref{rlt-closed-form} implies that \eqref{R1} is exact if and only if the minimum entry of $Q$ is on the diagonal. In this case, \eqref{stqp} has a 1-sparse optimal solution, i.e., the optimal solution of \eqref{stqp} without any sparsity constraint is already the sparsest possible solution. Furthermore, by Lemma~\ref{simple-obs1}, we immediately obtain
\begin{equation} \label{min-diag}
\ell(Q) = \ell_n(Q) = \ell_{n-1}(Q) = \ldots = \ell_1(Q) = \min\limits_{1\le k \le n} Q_{kk}.  \end{equation}

By reformulating the binarity constraint $u_j  \in  \{0,1\}$ with $u_j^2 = u_j,~j = 1,\ldots,n$ in (\ref{sstqp}), we obtain the following RLT relaxation:

\[
\begin{array}{llrcl}
\tag{R1($\rho$)}\label{R1rho} & \ell^{R1}_\rho (Q) :=  \min\limits_{x \in \R^n, u \in \R^n, X\in \cS^n,U \in \cS^n, R \in \R^{n \times n}} & \langle Q,X \rangle & & \\
 & \textrm{s.t.} & & & \\
 & & e\T x &= &1\\
 & & e\T u &=& \rho \\
 & & x & \leq & u \\
 & & x & \geq & 0 \\ 
  & & \textrm{diag} (U) & = & u\\
   & & Xe & = & x \\
   & & R\T e & = & u \\
   & & Re & = & \rho \, x \\
    & & Ue & = & \rho \, u \\
    & & X - R\T  - R + U & \geq & 0 \\
    & & X - R\T  & \leq & 0 \\
    & & R - U & \leq & 0 \\
    & & X , R, U & \geq & 0. \\
\end{array}
\]

Before we continue, let us remark that the constraints $x \leq u$ and $x \geq 0$ are redundant in (R1($\rho$)) since they are implied by the constraints $Xe = x$, $X \geq 0$, $R\T  e = u$, and $X - R\T  \leq 0$. Likewise, they imply $u\geq 0$ and $R \geq 0$. Furthermore, it is easy to verify that $R - U \leq 0$ and $U \geq 0$ are implied by the constraints $X - R\T  \leq 0$ and $X - R\T  - R + U\geq 0$. Note that $u\le e$ is not implied in this formulation.

Let us denote the projection of the feasible region of \eqref{R1rho} onto $(x,X)$ by
\begin{equation} \label{def_F_R1_rho}
\cF^{R1}_\rho := \left\{(x,X) \in \R^n \times \cS^n: (x,u,X,U,R) \textrm{ is \eqref{R1rho}-feasible for some } (u,U,R) \in \R^n \times \cS^n \times \R^{n \times n}\right\}\,  .
\end{equation}
Note that 
\begin{equation} \label{ell_rho_R1_alt}
\ell^{R1}_\rho (Q) =  \min\limits_{(x,X) \in \R^n \times \cS^n}\left\{\langle Q, X \rangle: (x,X) \in \cF^{R1}_\rho\right\}.    
\end{equation}

Clearly, we have $\cF^{R1}_\rho \subseteq \cF^{R1}$ if $1\le\rho \le n$, where $\cF^{R1}$ is given by \eqref{def_F_R1}. Our next result gives a description of $\cF^{R1}_\rho$ in closed form for each $\rho \in \{1,\ldots,n\}$.

\begin{lemma} \label{F_R1_rho_desc}
\begin{enumerate}
    \item[(i)] $\cF^{R1}_1 = \left\{(x,X) \in \R^n \times \cS^n: e\T  x = 1, \quad X = \textrm{Diag}(x), \quad x \geq 0 \right\}$.
    \item[(ii)] For each $\rho \in \{2,3,\ldots,n\}$, we have $\cF^{R1}_\rho = \cF^{R1}$, where $\cF^{R1}$ is given by \eqref{def_F_R1}.
\end{enumerate}  
\end{lemma}
\begin{proof}
   (i) Let $(  x,   X) \in \cF^{R1}_1$. Then $e\T x=1$ and $x\geq 0$. Moreover, there exists $(  u,   U,   R) \in \R^n \times \cS^n \times \R^{n \times n}$ such that $(  x,   u,   X,   U,   R)$ is  \eqref{R1rho}-feasible with $\rho=1$.
Since $  U \geq 0, ~\textrm{diag}(  U) =   u$, and $  U e =   u$, we obtain $  U = \textrm{Diag}(  u)$. By $  R -   U \leq 0$ and $  R \geq 0$, we obtain that $  R$ is a diagonal matrix. Similarly, using $  X -   R\T  \leq 0$, we conclude that $  X$ is a diagonal matrix. Since $  X e =   x$, we obtain that $  X = \textrm{Diag}(  x)$. Conversely, if $e\T    x = 1$, $  X = \textrm{Diag}(  x)$, and $  x \geq 0$, then it is easy to verify that $(  x,   u,   X,   U,   R) = (  x,   x,   X,   X,   X)$ is \eqref{R1rho}-feasible. It follows that $(  x,   X) \in \cF^{R1}_1$. \\[0.3em]
(ii) Let $\rho \in \{2,3,\ldots,n\}$. We clearly have $\cF^{R1}_\rho \subseteq \cF^{R1}$. Evidently, $ \cF^{R1}$ is a bounded polyhedron/polytope, so for the reverse inclusion, it suffices to show that each vertex of $\cF^{R1}$ belongs to $\cF^{R1}_\rho$. By Proposition~\ref{rlt-closed-form}, the set of vertices of $\cF^{R1}$ is given by \eqref{rlt-vertices}. If $(  x,  X) = (e^i, e^i (e^i)\T )$ for some $i = 1,\ldots,n$, then choose an arbitrary $  u \in \{0,1\}^n$ such that $  u_i = 1$ and $e\T    u = \rho$. If, on the other hand, $(  x,  X) = (\frac{1}{2}(e^i + e^j), \frac{1}{2}(e^i (e^j)\T +e^j (e^i)\T ))$ for some $1 \leq i < j \leq n$, then choose an arbitrary $  u \in \{0,1\}^n$ such that $  u_i = 1$, $  u_j = 1$, and $e\T    u = \rho$. In both cases then, define $  R =   x   u\T $ and $  U =   u   u\T $. It is easy to verify that $(  x,   u,   X,   U,   R) \in \R^n \times \R^n \times \cS^n \times \cS^n \times \R^{n \times n}$ is  \eqref{R1rho}-feasible, which implies that each vertex of $\cF^{R1}$ belongs to $\cF^{R1}_\rho$. We conclude that $\cF^{R1}_\rho = \cF^{R1}$.
\end{proof}

Lemma~\ref{F_R1_rho_desc} immediately gives rise to the following results.

\begin{corollary} \label{rlt-all-results}
\begin{enumerate}
    \item[(i)] For $\rho = 1$, \eqref{R1rho} is exact, i.e., $\ell^{R1}_1 (Q) = \ell_1(Q)$. 
    \item[(ii)] For each $\rho \in \{2,3,\ldots,n\}$, we have $\ell^{R1} (Q) = \ell^{R1}_\rho (Q) = \min\limits_{1\leq i \leq j \leq n} Q_{ij}$.
\end{enumerate}    
\end{corollary}
\begin{proof}
Both assertions follow from Lemma~\ref{F_R1_rho_desc}, Lemma~\ref{simple-obs1}, and \eqref{ell_rho_R1_alt}.  
\end{proof}

We arrive at the following exactness result for the classical RLT relaxation of sparse StQPs:

\begin{theorem} \label{exact-rlt}
\eqref{R1rho} is exact (i.e., $\ell^{R1}_\rho (Q) = \ell_\rho(Q)$) if and only if $\rho = 1$ or $\min\limits_{1\leq i \leq j \leq n} Q_{ij} = \min\limits_{1\le k \le n} Q_{kk}$. 
\end{theorem}
\begin{proof}
By Corollary~\ref{rlt-all-results}(i), \eqref{R1rho} is exact for $\rho = 1$. Let $\rho \in \{2,3,\ldots,n\}$. If \eqref{R1rho} is exact, then Lemma~\ref{simple-obs1} and Corollary~\ref{rlt-all-results}(ii) imply that $\ell^{R1}_\rho (Q) = \min\limits_{1\leq i \leq j \leq n} Q_{ij} = \ell^{R1} (Q) \leq \ell(Q) \leq \ell_\rho (Q) = \ell^{R1}_\rho (Q) = \ell^{R1} (Q)$. The claim follows from Proposition~\ref{rlt-closed-form}. Conversely, if $\min\limits_{1\leq i \leq j \leq n} Q_{ij} = \min\limits_{1\le k \le n} Q_{kk}$, then $\ell^{R1}_\rho (Q) = \ell^{R1} (Q) = \min\limits_{1\leq i \leq j \leq n} Q_{ij} = \min\limits_{1\le k \le n} Q_{kk} = \ell(Q) = \ell_\rho (Q)$ by Lemma~\ref{simple-obs1}, Corollary~\ref{rlt-all-results}(ii), and Proposition~\ref{rlt-closed-form}. Therefore, \eqref{R1rho} is exact.   
\end{proof}

By Theorem~\ref{exact-rlt}, \eqref{R1rho} is exact if and only if $\rho = 1$ or (\ref{stqp}) itself already has a 1-sparse optimal solution. Otherwise, in view of Lemma~\ref{simple-obs1} and the relation $\ell^{R1} (Q) \leq \ell(Q)$, it follows from Corollary~\ref{rlt-all-results} that, for each $\rho \geq 2$, the lower bound $\ell^{R1}_\rho (Q)$ arising from \eqref{R1rho} is, in general, quite weak as it already agrees with the lower bound $\ell^{R1} (Q)$ obtained from the RLT relaxation \eqref{R1} of (\ref{stqp}). 

\subsection{SDP Relaxation} \label{SDP-Relaxation}

In this section, we consider the standard Shor relaxation of (\ref{sstqp}) in relation to that of \eqref{stqp}. 

The Shor relaxation of (\ref{stqp}) is given by
\[
\tag{R2} \label{R2} \quad \ell^{R2}(Q) := \inf\limits_{x \in \R^n, X \in \cS^n} \left\{\langle Q, X \rangle: (x,X) \in \cF^{R2}\right\},
\]
where
\begin{equation} \label{def_F_R2}
\cF^{R2} := \left\{(x,X) \in \R^n \times \cS^n: e\T  x = 1, \quad x \geq 0, \quad X \succeq x x\T  \right\}\, , 
\end{equation}
a closed convex set not necessarily bounded, which necessitates the use of `$\inf$' in~\eqref{R2}.
Indeed, we have the following well-known result about (R2); we include a short proof for the sake of completeness.

\begin{lemma} \label{shor-stqp}
If $Q \succeq 0$, then (R2) is exact (i.e., $\ell^{R2}(Q) = \ell (Q)$). If $Q \not \succeq 0$, then $\ell^{R2}(Q) = -\infty$.    
\end{lemma}
\begin{proof}
If $Q \succeq 0$, then, for any (R2)-feasible solution $(   x,    X) \in \R^n \times \cS^n$, we have $\langle Q,    X \rangle \geq    x\T  Q    x$ since $   X \succeq    x    x\T $, which implies that $\ell(Q) \geq \ell^{R2}(Q) \geq \ell (Q)$. If $Q \not \succeq 0$, then there exists $   d \in \R^n$ such that $   d\T  Q    d < 0$. Let $   x \in \R^n$ be any feasible solution of (\ref{stqp}) and 
let $   X (\lambda) =    x    x\T  + \lambda    d    d\T $, where $\lambda \geq 0$. The assertion follows by observing that $(   x,    X (\lambda)) \in \cF^{R2}$ for each $\lambda \geq 0$ and that the objective function of (R2) evaluated at $(   x,    X (\lambda))$ tends to $-\infty$ as $\lambda \to \infty$.      
\end{proof}

The Shor relaxation of (\ref{sstqp}) is given by
\[
\begin{array}{llrcl}
\tag{R2($\rho$)}\label{R2rho} & \ell^{R2}_\rho(Q) := \min\limits_{x \in \R^n, u \in \R^n, X \in \cS^n, U \in \cS^n, R \in \R^{n\times n}} & \langle Q,X \rangle & & \\
 & \textrm{s.t.} & & & \\
 & & e\T x & = &1\\
 & & e\T  u & = & \rho\\
 & & \textrm{diag}(U) & = & u\\
 & & x & \leq & u \\
 & & x & \geq & 0 \\
 & & \begin{bmatrix}
    1 & x\T  & u\T \\
    x & X   & R  \\
    u & R\T  & U  \\
    \end{bmatrix} & \succeq & 0.  
\end{array}
\]
Note that the constraint $u \leq e$ is implied by $\textrm{diag}(U) = u$ and the semidefiniteness constraint. 

Similar to the RLT relaxation of \eqref{sstqp}, let us introduce the following projection of the feasible region of \eqref{R2rho} onto $(x,X)$:
\begin{equation} \label{def_F_R2_rho}
\cF^{R2}_\rho := \left\{(x,X) \in \R^n \times \cS^n:  (x,u,X,U,R) \textrm{ is \eqref{R2rho}-feasible for some } (u,U,R) \in \R^n \times \cS^n \times \R^{n \times n}\right\}\, .
\end{equation}
We again observe that  
\begin{equation} \label{ell_rho_R2_alt}
\ell^{R2}_\rho (Q) =  \min\limits_{(x,X) \in \R^n \times \cS^n}\left\{\langle Q, X \rangle: (x,X) \in \cF^{R2}_\rho\right\}.    
\end{equation}

Our next result gives a complete description of $\cF^{R2}_\rho$ for each $\rho = 1,2,\ldots,n$.

\begin{lemma} \label{F_R2_rho_desc}
For each $\rho \in \{1,2,\ldots,n\}$, we have $\cF^{R2}_\rho = \cF^{R2}$, where $\cF^{R2}$ is given by \eqref{def_F_R2}.     
\end{lemma}
\begin{proof}
We clearly have $\cF^{R2}_\rho \subseteq \cF^{R2}$. For the reverse inclusion, let $(  x,   X) \in \cF^{R2}$ so that $e\T x=1$ and $x\geq 0$, so also $x\leq e$. Furthermore $  X =   x   x\T  +   M$ for some $  M \succeq 0$. Define $  u =   x + \left( \frac{\rho - 1}{n-1}\right) (e -   x)$ so that $e\T    u = \rho$ and $0 \leq   x \leq   u \leq e$.  Let $  R =   x   u\T $ and $  U =   u  u\T  +   D$, where $  D \in \cS^n$ is a diagonal matrix such that $  D_{jj} =   u_j - (  u_j)^2 \geq 0,~j = 1,\ldots,n$. Note that $\textrm{diag}(  U) =   u$ and 
\[
\begin{bmatrix}
                              X   &   R  \\
                              R\T  &   U
                            \end{bmatrix} - \begin{bmatrix}   x \\   u \end{bmatrix} \begin{bmatrix}   x \\   u \end{bmatrix}\T  = \begin{bmatrix}
                              M   & 0  \\
                            0 &   D
                            \end{bmatrix} \succeq 0.
\] 
By Schur complementation, it follows that $(  x,   u,   X,   U,   R) \in \R^n \times \R^n \times \cS^n \times \cS^n \times \R^{n\times n}$ is  (R2($\rho$))-feasible. Therefore  $(  x,   X) \in \cF^{R2}_\rho$.
\end{proof}

Lemma~\ref{F_R2_rho_desc} reveals that none of the feasible solutions of $\cF^{R2}$ is cut off in the projection of the feasible region of \eqref{R2rho} for any choice of $\rho \in \{1,2,\ldots,n\}$. In view of \eqref{ell_rho_R2_alt}, we obtain the following corollary. 

\begin{corollary} \label{shor-sstqp}
For any $\rho \in \{1,\ldots,n\}$, we have $\ell^{R2}_\rho(Q) = \ell^{R2}(Q)$.  
\end{corollary} 
\begin{proof}
The assertion follows from \eqref{R2}, \eqref{ell_rho_R2_alt}, and Lemma~\ref{F_R2_rho_desc}.
\end{proof}

Now we obtain the following exactness result for the Shor relaxation of the sparse StQP:

\begin{theorem} \label{exact-shor-sstqp}
\eqref{R2rho} is exact (i.e., $\ell^{R2}_\rho(Q) = \ell_\rho(Q)$) if and only if $Q \succeq 0$ and (\ref{stqp}) has a $\rho$-sparse optimal solution. 
\end{theorem}
\begin{proof}
The assertion follows from Lemma~\ref{shor-stqp}, Corollary~\ref{shor-sstqp}, and Lemma~\ref{simple-obs1}.  
\end{proof}

Theorem~\ref{exact-shor-sstqp} shows that \eqref{R2rho} provides a finite lower bound if and only if $Q \succeq 0$. Furthermore, in this case, the bound is tight if and only if the problem \eqref{stqp} without any sparsity constraint already has a $\rho$-sparse optimal solution. It follows that \eqref{R2rho}, in general, is a weak relaxation. We close this section by specializing Theorem~\ref{exact-shor-sstqp} to the particular case with a rank-one $Q \in \cS^n$.

\begin{corollary} \label{rank-one-shor}
Let $Q = vv\T $, where $v \in \R^n$. If $v \in \R^n_+$ or $-v \in \R^n_+$ or $v_i = 0$ for some $i \in \{1,\ldots,n\}$, then $\ell^{R2}_\rho(Q) = \ell_\rho(Q)$ for each $\rho\in \{1,\ldots,n\}$. Otherwise,  $\ell^{R2}_1(Q) < \ell_1(Q) $ and $\ell^{R2}_\rho(Q) = \ell_\rho(Q)$ for each $\rho \in \{ 2,\ldots,n\}$.
\end{corollary}
\begin{proof}
Let $Q = vv\T $, where $v \in \R^n$. Note that $x\T  Q x = (v\T  x)^2 \geq 0$ for each $x \in \R^n$. If $v \in \R^n_+$ (resp., $-v \in \R^n_+$ or $v_i = 0$ for some $i \in \{1,\ldots,n\}$), then (\ref{stqp}) has a 1-sparse optimal solution given by $e^j \in \R^n$, where $j = \arg \min\limits_{1\leq i \leq n} v_i$ (resp., $j = \arg \min\limits_{1\leq i \leq n}(-v_i)$ or $i = j$). The assertion follows from Theorem~\ref{exact-shor-sstqp}. Otherwise, there exist $i \in \{1,\ldots,n\}$ and $j \in \{1,\ldots,n\}$ such that $v_i < 0 < v_j$. 
Therefore, setting $x= \frac{v_j}{v_j-v_i}\, e^i - \frac{v_i}{v_j-v_i}\, e^j$,
we obtain $  x \in F$ and $  x\T  Q   x = (v\T    x)^2 = 0 = \ell(Q)$. On the other hand $\ell_1 (Q) = \min\limits_{1\le k \le n}Q_{kk} = \min\limits_{1\le k \le n}v_{k}^2 > 0$ by Lemma~\ref{simple-obs1}.     
\end{proof}

A comparison of Corollary~\ref{rank-one-shor} and Theorem~\ref{exact-rlt} reveals that the Shor relaxation \eqref{R2rho} can be strictly weaker than the RLT relaxation \eqref{R1rho} for $\rho = 1$, even when $Q \succeq 0$. 

\section{SDP-RLT Relaxation} \label{SDP-RLT-Relaxation}

In this section, we consider the SDP-RLT relaxations of (\ref{sstqp}) and \eqref{stqp} obtained by combining the corresponding RLT relaxations and SDP relaxations presented in Section~\ref{RLT-Relaxation} and Section~\ref{SDP-Relaxation}, respectively. In particular, our objective is to shed light on the properties of the combined relaxation in relation to those of the two individual relaxations.

The SDP-RLT relaxation of (\ref{stqp}) is given by
\[
\tag{R3} \label{R3} \quad \ell^{R3}(Q) := \min\limits_{x \in \R^n, X \in \cS^n} \left\{\langle Q, X \rangle: (x,X) \in \cF^{R3}\right\},
\]
where
\begin{equation} \label{def_F_R3}
\cF^{R3} := \left\{(x,X) \in \R^n \times \cS^n: e\T  x = 1, \quad X e = x, \quad x \geq 0, \quad X \geq 0, \quad X \succeq x x\T  \right\}.  
\end{equation}

A complete description of instances of \eqref{stqp} that admit exact SDP-RLT relaxations is given below.

\begin{theorem}
[G{\"o}kmen and \Yildirim~{\rm\cite{Goek22}}]
\label{exact_dnn_stqp}
\eqref{R3}
is exact (i.e., $\ell^{R3}(Q) = \ell(Q)$) if and only if (i) $n \leq 4$; or (ii) $n \geq 5$ and there exist $   x \in F$, $   P \succeq 0$, $N \in \cS^n$, $   N \geq 0$, $   \lambda \in \R$ such that $   P    x = 0$, $   x\T     N    x = 0$, and $Q =    P +    N +    \lambda E$. Furthermore, for any such decomposition, $   x \in F$ is an optimal solution of (\ref{stqp}) and $\ell^{R3}(Q) = \ell(Q) =    \lambda$.
\end{theorem}

We next consider the SDP-RLT relaxation of (\ref{sstqp}):

\[
\begin{array}{llrcl}
\tag{R3($\rho$)} \label{R3rho} &\ell^{R3}_\rho(Q) :=  \min\limits_{x \in \R^n, u \in \R^n, X\in \cS^n,U \in \cS^n, R \in \R^n} & \langle Q,X \rangle & & \\
 & \textrm{s.t.} & & & \\
 & & e\T x &= &1\\
 & & e\T u &=& \rho \\
 & & x & \leq & u \\
 & & x & \geq & 0 \\ 
  & & \textrm{diag}(U) & = &  u \\
   & & Xe & = & x \\
   & & R\T e & = & u \\
   & & Re & = & \rho \, x \\
    & & Ue & = & \rho \, u \\
    & & X - R\T  - R + U & \geq & 0 \\
    & & X - R\T  & \leq & 0 \\
    & & R - U & \leq & 0 \\
    & & X, R, U & \geq & 0 \\
    & & \begin{bmatrix}1 & x\T  & u\T  \\
     x& X&R \\
     u&R\T &U\end{bmatrix}& \succeq & 0. \\
\end{array}
\]

Similar to the RLT and SDP relaxations, consider the projection of the feasible region of \eqref{R3rho} onto $(x,X)$ given by
\begin{equation} \label{def_F_R3_rho}
\cF^{R3}_\rho = \left\{(x,X) \in \R^n \times \cS^n: (x,u,X,U,R) \textrm{ is \eqref{R3rho}-feasible for some }(u,U,R) \in \R^n \times \cS^n \times \R^{n \times n} \right\}.
\end{equation}
In a similar manner as above, we have  
\begin{equation} \label{ell_rho_R3_alt}
\ell^{R3}_\rho (Q) =  \min\limits_{(x,X) \in \R^n \times \cS^n}\left\{\langle Q, X \rangle: (x,X) \in \cF^{R3}_\rho\right\}\, .    
\end{equation}
It is also easy to see that
\begin{equation} \label{proj_rels_1}
\cF^{R3}_\rho \subseteq \cF^{R1}_\rho \cap \cF^{R2}_\rho\, ,  
\end{equation}
where $\cF^{R1}_\rho$ and $\cF^{R2}_\rho$ are given by \eqref{def_F_R1_rho} and \eqref{def_F_R2_rho}, respectively. Therefore, 
\begin{equation} \label{rels_R1_R2_R3}
\max\{\ell^{R1}_\rho(Q), \ell^{R2}_\rho(Q)\} \leq \ell^{R3}_\rho(Q) \leq \ell_\rho(Q) \quad \textrm{for all }\rho \in \{ 1,\ldots,n\}\, , 
\end{equation}
which implies that \eqref{R3rho} is at least as tight as each of \eqref{R1rho} and \eqref{R2rho}.

Our first result follows from the previous results on weaker relaxations.

\begin{corollary} \label{sdp-rlt-exact1}
If (i) $\rho = 1$, or (ii) $\min\limits_{1\leq i \leq j \leq n} Q_{ij} = \min\limits_{1\le k \le n} Q_{kk}$, or (iii) $Q \succeq 0$ and (\ref{stqp}) has a $\rho$-sparse optimal solution, then \eqref{R3rho} is exact.
\end{corollary}
\begin{proof}
The assertion follows from Theorem~\ref{exact-rlt}, Theorem~\ref{exact-shor-sstqp}, and \eqref{rels_R1_R2_R3}.    
\end{proof}

\subsection{Projected Feasible Sets and Their Inner Approximations}

We now focus on the sets $\cF^{R3}_\rho$, $\rho \in \{ 1,\ldots,n\}$. By Lemma~\ref{F_R1_rho_desc}, Lemma~\ref{F_R2_rho_desc}, \eqref{def_F_R3}, and \eqref{proj_rels_1}, 
\begin{eqnarray} \label{proj_rels_2}
 \cF^{R3}_1 &\subseteq &\left\{(x,X) \in \R^n \times \cS^n: e\T  x = 1, \, X = \textrm{Diag}(x), \, X \succeq x x\T , \, x \geq 0\right\} \subseteq \cF^{R3}, \label{proj_rels_2a}\\
 \cF^{R3}_\rho &\subseteq &\left\{(x,X) \in \R^n \times \cS^n: e\T  x = 1, \, Xe = x, \, X \succeq x x\T , \, X \geq 0, \,x \geq 0\right\} = \cF^{R3}, \;\rho \geq 2\,. \label{proj_rels_2b}
\end{eqnarray}

Next, we consider inner approximations of the sets $\cF^{R3}_\rho$, where $\rho \in \{1,2,\dots,n\}$. 

\begin{proposition} \label{inner-rank-one}
   For any fixed $\rho\in \{ 1,\ldots,n\}$, 
   consider the corresponding formulation \eqref{sstqp}. Then, 
   we have 
\begin{equation} \label{proj_rels_4}
\textrm{conv}\left\{(x,x x\T ): x \in F_\rho\right\} \subseteq \cF^{R3}_\rho \, , 
\end{equation} 
where $F_\rho$ and $\cF^{R3}_\rho$ are given by \eqref{def_F_rho} and \eqref{def_F_R3_rho}, respectively.
\end{proposition} 
\begin{proof}
For any \eqref{sstqp}-feasible solution $(   x,   u) \in \R^n \times \R^n$, we define $   X =    x    x\T $, $   R =    x    u\T $, and $   U =    u    u\T $. Then obviously
$(   x,    u,    X,    U,    R) \in \R^n \times \R^n \times \cS^n \times \cS^n \times\R^{n \times n} $ is \eqref{R3rho}-feasible. The claim now follows by~\eqref{def_F_R3_rho} and the convexity of $\cF^{R3}_\rho$. 
\end{proof}

In the remainder of this section, we identify further properties of the sets $\cF^{R3}_\rho$, where $\rho \in \{1,2,\dots,n\}$ and their implications on the tightness of the lower bound $\ell^{R3}_\rho(Q)$.

\subsection{The Extremely Sparse Case $\rho=1$}

In this section, we give an exact description of the set $\cF^{R3}_1$ and discuss its implications. We start with a technical lemma.

\begin{lemma}\label{psd-condition}
    For any $a \in \R^n_+$ such that $e\T a  \leq 1$, we have $ \textrm{Diag}(a) - aa\T  \succeq 0$. 
\end{lemma}
\begin{proof}
Let $a_\mathbf{P}$ denote the subvector of $a$ with strictly positive components. Note that $\textrm{Diag}(a) - aa\T  \succeq 0$ if and only if $\textrm{Diag}(a_\mathbf{P}) - a_\mathbf{P}a_\mathbf{P}\T  \succeq 0$. Therefore, without loss of generality, we may and do assume that $a = a_\mathbf{P}$. We have $\textrm{Diag}(a) - aa\T  \succeq 0$ if and only if $\textrm{Diag}(\sqrt{a})^{-1}  (\textrm{Diag}(a) - aa\T ) \textrm{Diag}(\sqrt{a})^{-1} = I - \sqrt{a}\sqrt{a}\T  \succeq 0$, where $\sqrt{a} := [\sqrt{a_1},\sqrt{a_2},\ldots,\sqrt{a_n}]\T $. The only nonzero eigenvalue of the rank-one matrix $\sqrt{a}\sqrt{a}\T $ is $\sqrt{a}\T \sqrt{a} = e\T  a \leq 1$. Therefore, $ I - \sqrt{a}\sqrt{a}\T   \succeq 0$, which implies $\textrm{Diag}(a) - aa\T  \succeq 0$. 
\end{proof}

By Lemma~\ref{psd-condition}, it is easy to see that the constraint $X - x x\T  \succeq 0$ on the right-hand side of \eqref{proj_rels_2a} is redundant. Therefore, by Lemma~\ref{F_R1_rho_desc}, we obtain
\begin{equation} \label{proj_rels_3}
\cF^{R3}_1 \subseteq \left\{(x,X) \in \R^n \times \cS^n: e\T  x = 1, \quad X = \textrm{Diag}(x), \quad x \geq 0\right\} = \cF^{R1}_1.    
\end{equation}

Our next result shows that the inclusion in \eqref{proj_rels_3} actually holds with equality, thereby yielding an exact description of $\cF^{R3}_1$.

\begin{lemma} \label{F_R3_1_desc}
We have
\begin{equation} \label{proj_rels_5}
\cF^{R3}_1 = \cF^{R1}_1 = \textrm{conv}\left\{(x,x x\T ): x \in F_1 \right\} = \textrm{conv}\left\{(e^j,e^j (e^j)\T ): j \in \{ 1,\ldots,n \}\right\},     
\end{equation}
where $\cF^{R3}_1$ and $\cF^{R1}_1$ are defined as in \eqref{def_F_R3_rho} and \eqref{def_F_R1_rho}, respectively.
\end{lemma}
\begin{proof}
The assertion follows from the observation that $\cF_1^{R1} = \textrm{conv}\left\{(e^j,e^j (e^j)\T ): 1\le j \le n \right\}$ in conjunction with Proposition~\ref{inner-rank-one} and~\eqref{proj_rels_3}. 
\end{proof}

Lemma~\ref{F_R3_1_desc} reveals that the SDP-RLT relaxation \eqref{R3rho} is identical to the RLT relaxation \eqref{R1rho} for $\rho = 1$: semidefinite constraints in \eqref{R3rho} are redundant.

\subsection{Case of Larger Sparsity $\rho \geq 2$}

In this section, we focus on the sets $\cF^{R3}_\rho$, where $\rho \in \{2,3,\ldots,n\}$, and establish several properties and relations. Our first result strengthens the inner approximation of $\cF^{R3}_\rho$ given by Proposition~\ref{inner-rank-one}.

\begin{lemma} \label{comp1-R3-R3rho}
We have 
\[
\left\{(x,X) \in \cF^{R3}: x \in F_\rho\right\} \subseteq \cF^{R3}_\rho, \quad \textrm{all }\rho \in \{ 2,\ldots,n\}\, ,
\]
where $F_\rho$, $\cF^{R3}$, and $\cF^{R3}_\rho$ are given by \eqref{def_F_rho}, \eqref{def_F_R3} and \eqref{def_F_R3_rho}, respectively.
\end{lemma}
\begin{proof}
Fix $\rho \in \{2,\ldots,n\}$ and let $(  x,   X) \in \cF^{R3}$ with $\|   x \|_0 \leq \rho$. Choose $  u \in \{0,1\}^n$ such that $  x \leq   u$ and $e\T    u = \rho$. Define $  R =   x   u\T $ and $  U =   u   u\T $. Clearly, $\textrm{diag}(  U) =   u$, $  R\T  e =   u,   R e = \rho   x$, $  U e = \rho   u$, $  R -   U \leq 0$, $  R \geq 0$, and $  U \geq 0$. Since $  X =   x   x\T  +   M$ for some $  M \succeq 0$, we obtain
\[
\begin{bmatrix}
  X   &   R  \\
  R\T  &   U
\end{bmatrix} - \begin{bmatrix}   x \\   u \end{bmatrix} \begin{bmatrix}   x \\   u \end{bmatrix}\T  = \begin{bmatrix}
  M   & 0  \\
0 & 0
\end{bmatrix} \succeq 0.
\] 

Next, we consider the constraint $  X -   R\T  \leq 0$. Since $  X \geq 0$ and $  X e =   x$, we obtain $0 \leq   X_{ij} \leq \min\{  x_i,  x_j\}$ for each $1 \leq i \leq j \leq n$. Therefore, if $\min\{  x_i,  x_j\} = 0$, then $  X_{ij} -   u_i   x_j = -   u_i   x_j \leq 0$. On the other hand, if $\min\{  x_i,  x_j\} > 0$, then $  u_i = 1$, which implies that $  X_{ij} -   u_i   x_j =   X_{ij} -   x_j \leq 0$. It follows that $  X -   R\T  \leq 0$.

Finally, we need to show that $  X -   R -   R\T  +   U \geq 0$. For each $1 \leq i \leq j \leq n$, if $\min\{  x_i,  x_j\} = 0$, then $  X_{ij} = 0$ and $\min\{  R_{ij},  R_{ji}\} = \min\{  x_i   u_j,  x_j   u_i\} = 0$. Therefore,
$$  X_{ij} -   R_{ij} -   R_{ji} +   U_{ij} = 0-\max\{ R_{ij},  R_{ji}\}-0 + u_iu_j = - \max\{  x_i   u_j,  x_j   u_i\} +   u_i   u_j \geq 0\, ,$$ 
since $  x \leq   u$. Here, we used the lattice identity $v+w=\min \{ v,w\} + \max\{v,w\}$. On the other hand, if $\min\{  x_i,  x_j\} > 0$, then $  u_i =   u_j = 1$, which implies that $  X_{ij} -   R_{ij} -   R_{ji} +   U_{ij} =   X_{ij} -   x_i -   x_j + 1$. For any $1 \leq i < j \leq n$, since $  x_i +   x_j \leq 1$, we clearly have $  X_{ij} -   x_i -   x_j + 1 \geq 0$ since $  X \geq 0$. Finally, if $i = j$, since $  X \succeq 0$ and $  X e =   x$, we obtain 
\[
X_{ii} - 2   x_i + 1 = ( e^i - e)\T    X ( e^i - e)  \geq 0, \quad  i = 1,\ldots,n,
\]
which completes the proof.
\end{proof}

By Lemma~\ref{comp1-R3-R3rho}, none of the solutions in $(x,X) \in \cF^{R3}$ with $x \in F_\rho$ is cut off by the projection $\cF^{R3}_\rho$. This observation gives rise to the following corollary.

\begin{corollary} \label{stqp_sparse_x}
\begin{enumerate}
    \item[(i)] For each $\rho \in \{2,\ldots,n\}$, if there exists an optimal solution $(  x,   X) \in \R^n \times \cS^n$ of \eqref{R3} such that $\|   x \|_0 \leq \rho$, then $\ell^{R3}(Q) = \ell^{R3}_\rho(Q)$. 
    \item[(ii)] We have $\cF^{R3}_n = \cF^{R3}$ and $\ell^{R3}(Q) = \ell^{R3}_n(Q)$.
\end{enumerate} 
\end{corollary}
\begin{proof}
\begin{enumerate}
    \item[(i)] We clearly have $\ell^{R3}(Q) \leq \ell^{R3}_\rho(Q)$ by \eqref{R3}, \eqref{ell_rho_R3_alt}, and \eqref{proj_rels_2b}. The reverse inequality follows from Lemma~\ref{comp1-R3-R3rho}. 
    \item[(ii)] As $F_n=F$, the first equality follows from \eqref{proj_rels_2b} and Lemma~\ref{comp1-R3-R3rho}, and the second one from the first assertion~(i).
\end{enumerate}
\end{proof}

By Corollary~\ref{stqp_sparse_x}, we can identify a particular set of instances of \eqref{sstqp} that admit an exact SDP-RLT relaxation.

\begin{corollary} \label{exact-sdp-rlt-1}
Let $\rho \in \{2,\ldots,n\}$. For any $  x \in F_\rho$, any $  P \succeq 0$ such that $  P   x = 0$, any $N \in \cS^n$ such that $  N \geq 0$ and $  x\T    N   x = 0$, and any $  \lambda \in \R$, if $Q =   P +   N +   \lambda E$, then the SDP-RLT relaxation \eqref{R3rho} is exact, i.e., $\ell^{R3}_\rho(Q) = \ell_\rho(Q)$.     
\end{corollary}
\begin{proof}
Under the hypotheses, Theorem~\ref{exact_dnn_stqp} implies that $  x \in F$ is an optimal solution of (\ref{stqp}) and $\ell^{R3}(Q) = \ell(Q) =   \lambda$. The assertion follows from Corollary~\ref{stqp_sparse_x}(i) and Lemma~\ref{simple-obs1}.  
\end{proof}

\subsection{Rank-One Elements of $\cF^{R3}_\rho$}

Recall that each solution $(x,X) \in \cF^{R3}$, where $x \in F_\rho$, is retained in the projection $\cF^{R3}_\rho,~\rho = 1,\ldots,n$ by Lemma~\ref{comp1-R3-R3rho}. In this section, our goal is to shed light on the relations between $\cF^{R3}_\rho$ and the set of solutions $(x,X) \in \cF^{R3}$, where $\|x\|_0  > \rho$.

First, it follows from Proposition~\ref{inner-rank-one} and Lemma~\ref{F_R3_1_desc} that 
\begin{equation} \label{proj_rels_6}
\cF^{R3}_1 \subseteq \cF^{R3}_\rho \quad\mbox{for all } \rho \in \{ 2,\ldots,n\}\, ,    
\end{equation}
which, in turn, implies that $(x,X) = (\frac{1}{n} e, \frac{1}{n} I) \in \cF^{R3}_\rho$ for each $\rho \in \{1,\ldots,n\}$ by Lemma~\ref{F_R3_1_desc}. Therefore, for each $\rho \in \{1,\ldots,n\}$, there exists $(x,X) \in \cF^{R3}_\rho$ such that $\|x\|_0  > \rho$. 

Let us
restrict our attention to the subset of ``rank-one solutions'' $(x,X) \in \cF^{R3}$, i.e., those with $\|x\|_0 = \nu > \rho$ and $X = x x\T $. Note that $\langle Q, X \rangle = x\T  Q x$ for each rank-one solution. This, in turn, enables us to compare $\ell^{R3}_\rho(Q)$ and $\ell_{\nu}(Q)$ for some $\nu > \rho$.

We start with the following result for $\rho = 1$.

\begin{corollary} \label{rank-one-rho-1}
$( x,  x  x\T ) \in \cF^{R3}_1$ if and only if $x \in F_1$. 
\end{corollary}
\begin{proof}
The claim follows from Lemma~\ref{F_R3_1_desc}.
\end{proof}

By Corollary~\ref{rank-one-rho-1}, each rank-one solution $( x,  x  x\T ) \in \cF^{R3}$, where $\|  x \|_0 > 1$, is cut off by $\cF^{R3}_1$. We next focus on $\cF^{R3}_\rho$ for $\rho \geq 2$. To that end, we first state a technical result about the feasible region of (R3($\rho$)).

\begin{lemma} \label{u_lower_bound}
Let $( x,  u,  X,  U,  R) \in \R^n \times \R^n \times \cS^n \times \cS^n \times \R^{n\times n}$ be (R3($\rho$))-feasible, where $\rho \in \{1,\ldots,n\}$. Then,  
\begin{equation} \label{u_lb_1}
\left(\rho - 2\right)  u_i + 2  R_{ii} + (1 - \rho)  x_i -  X_{ii} \geq 0, \quad \mbox{for all }i \in \{1,\ldots,n\}\, .    
\end{equation}
\end{lemma}
\begin{proof}
Suppose that $( x,  u,  X,  U,  R) \in \R^n \times \R^n \times \cS^n \times \cS^n \times \R^{n\times n}$ is  (R3($\rho$))-feasible. Let us fix $i \in \{1,\ldots,n\}$. For each $j \in \{1,\ldots,n\}$ such that $j \neq i$, we have
\[
 U_{ij} -  R_{ij} -  R_{ji} +  X_{ij} \geq 0.
\]
Therefore,
\begin{eqnarray*}
    0 & \leq & \sum\limits_{j \in \{1,\ldots,n\} \backslash \{i\}} \left( U_{ij} -  R_{ij} -  R_{ji} +  X_{ij}\right) \\
    & = & \left(\rho  u_i -  u_i \right) - \left( \rho  x_i -  R_{ii} \right) - \left(  u_i -  R_{ii}\right) + \left( x_i -  X_{ii}\right)\\
    & = & \left(\rho - 2\right)  u_i + 2  R_{ii} + (1 - \rho)  x_i -  X_{ii}, 
\end{eqnarray*}
where we used $\textrm{diag}( U) =  u$, $ X e =  x$, $  R\T  e =  u$, $ R e = \rho \,  x$, and $ U e = \rho \,  u$ in the second line. The assertion follows.   
\end{proof}

Using this technical result, we can establish the following result about rank-one solutions for $\rho = 2$.

\begin{corollary} \label{ub-rankone-rho-2}
For each $ x \in F$ such that $\| x\|_0 \geq 4$, we have $( x,  x  x\T ) \not \in \cF^{R3}_2$. 
\end{corollary}
\begin{proof}
We prove the contrapositive. Let $\rho = 2$ and let $( x,  x  x\T ) \in \cF^{R3}_\rho$. Then, there exists $( u,  U,  R) \in \R^n \times \cS^n \times \R^{n \times n}$ such that $( x,  u,  X,  U,  R) \in \R^n \times \R^n \times \cS^n \times \cS^n \times \R^{n\times n}$ is (R3($\rho$))-feasible, where $ X =  x  x\T $. Since $ X =  x  x\T $, it follows from the positive semidefiniteness constraint that $ R =  x  u\T $. By Lemma~\ref{u_lower_bound}, we obtain
\[
\left(\rho - 2\right)  u_i + 2  x_i  u_i + (1 - \rho)  x_i -  x_i^2 \geq 0 \quad\mbox{for all } i \in \{ 1,\ldots,n\}\, .
\]
Using $\rho = 2$, for each $i \in \{1,\ldots,n\}$ such that $ x_i > 0$, we obtain
\[
 u_i \geq \frac{1 +  x_i}{2}\, .
\]
Summing over each $i \in \{1,\ldots,n\}$ such that $ x_i > 0$, and observing $\sum\limits_{i:x_i>0} x_i= e\T x= 1$, we arrive at
\[
2 = \sum_i u_i \ge\sum_{i: x_i>0} u_i \geq \frac{\| x\|_0 + 1}{2}\, ,
\]
which implies that $\| x \|_0 \leq 3$. The assertion follows.
\end{proof}

By Corollary~\ref{ub-rankone-rho-2}, each rank-one solution $( x,  x  x\T ) \in \cF^{R3}$, where $\|  x \|_0 > 3$, is cut off by $\cF^{R3}_2$. Furthermore, for each $ x \in F$ such that $\| x\|_0 = 3$, the proof of Corollary~\ref{ub-rankone-rho-2} implies that there exists a unique $ u \in \R^n$ given by $ u = \frac{1}{2}( x + e) =  x + \frac{1}{2}(e -  x)$ such that $( x,  u,  x  x\T ,  U,  R) \in \R^n \times \R^n \times \cS^n \times \cS^n \times \R^{n\times n}$ is (R3($\rho$))-feasible. Our next result establishes that the choice of $ u$ can be generalized to larger values of $\rho$.

\begin{theorem} \label{general-construct}
We have
\begin{eqnarray} \label{incls1-3}
\left\{(x, x x\T ) \in \cF^{R3}: x \in F_{2 \rho - 1}\right\} & \subseteq & \cF^{R3}_\rho \quad \mbox{for all }\rho\in \left\{ 2,\ldots,\left \lfloor \textstyle{\frac{n+1}{2}} \right \rfloor\right\}\, , \label{incl1} \\
\left\{(x, x x\T ) \in \cF^{R3}: x \in F\right\} & \subseteq & \cF^{R3}_\rho \quad \mbox{for all }\rho \in \left\{  \left \lfloor \textstyle{\frac{n+1}{2}} \right \rfloor + 1, \ldots,n\right\}\, , \label{incl2} \\
\left\{(x, x x\T ) \in \cF^{R3}: x \in G_\rho \right\} & \subseteq & \cF^{R3}_\rho \quad \mbox{for all }\rho \in \left\{ 2, \ldots, \left \lfloor \textstyle{\frac{n}{2}} \right \rfloor\right\}\, , \label{incl3}
\end{eqnarray}
where we define for $\rho \in \left\{ 2, \ldots, \left \lfloor \textstyle{\frac{n}{2}} \right \rfloor\right\}$
\begin{equation} \label{suff_cond_set1}
G_\rho := \left\{x \in F: \|x\|_0 > 2 \rho - 1, \quad \max\limits_{1 \leq i < j \leq n: x_i x_j > 0} \frac{x_i x_j}{1 - x_i - x_j} \leq \frac{(\rho - 1)(\rho - 2)}{\left(\|x\|_0 - 2\right) \left(\|x\|_0 - 2 \rho + 1\right)}\right\}\, .
\end{equation}
\end{theorem}
\begin{proof}
By Corollary~\ref{stqp_sparse_x}(ii), we have $\cF^{R3}_n = \cF^{R3}$, which implies \eqref{incl2} for $\rho = n$. Therefore, let $\rho \in \{2,\ldots,n-1\}$. By Lemma~\ref{comp1-R3-R3rho}, it suffices to 
focus on rank-one solutions $(x, x x\T )$, where $x\in F$ with $\| x\|_0 \geq \rho + 1$. We abbreviate $\nu:=\| x\|_0 $ to ease notation. Our proof is constructive.  Let us define $  u \in \R^n$ as follows:
\[
  u_i = \begin{cases}   x_i + \lambda (1 -   x_i)\, , 
 & \textrm{if}~  x_i > 0\, , \\
 0\, , & \textrm{otherwise,}
\end{cases}
\]
where 
\begin{equation} \label{def_lambda}
\lambda := \frac{\rho - 1}{\nu  - 1} \in (0,1).
\end{equation}
Note that $0 \leq   x \leq   u \leq e$ and $e\T    u = \rho$. Let us define $  X =   x   x\T $, $  R =   x   u\T $, and 
\[
  U =   u   u\T  + U^1 + U^2,
\]
where 
\begin{eqnarray*}
U^1 & := & \alpha\left(\textrm{Diag}(  x) -   x   x\T \right),\\
U^2 & := & \beta \left(\textrm{Diag}(a) - a a\T \right),
\end{eqnarray*}
and $\alpha$, $\beta$, and $a \in \R^n$ are given by
\begin{eqnarray} \label{def_params}
\alpha & := & \frac{(\nu  - \rho)(\nu  - \rho - 1)}{(\nu  - 1)(\nu  - 2)} \geq 0, \label{def_alpha} \\
\beta & := & \frac{(\nu  - \rho)(\rho - 1)}{\nu  - 2} > 0, \label{def_beta} \\
a_i & := &\begin{cases}
 \frac{1 -   x_i}{\nu  - 1}, & \textrm{if}~  x_i > 0, \\
 0, & \textrm{otherwise.}
 \end{cases} \label{def_a}
\end{eqnarray}
Note that $a \in \R^n_+$, $e\T  a = 1$, and $  U e = \rho   u$, because $U^1 e =U^2 e=0$. Furthermore,
\[
\begin{bmatrix}
                              X   &   R  \\
                              R\T  &   U
                            \end{bmatrix} - \begin{bmatrix}   x \\   u \end{bmatrix} \begin{bmatrix}   x \\   u \end{bmatrix}\T  = \begin{bmatrix}
                            0 & 0  \\
                            0 & U^1 + U^2
                            \end{bmatrix} \succeq 0,
\] 
where we used Lemma~\ref{psd-condition}. In addition, if $  x_i = 0$, then $  U_{ii} = 0 =   u_i$. If $  x_i > 0$, then it follows as well that $U_{ii}=u_i$, along the following lines:
 \begin{equation}\label{crux}\begin{array}{rcl}
  U_{ii} & = &   u_i^2 + \alpha (  x_i -   x_i^2) + \beta (a_i - a_i^2) 
\\
 & = & \left((1 - \lambda)   x_i + \lambda \right)^2 + \alpha (  x_i -   x_i^2) +  \frac{\beta}{\nu  - 1}  (1 -   x_i) -  \frac{\beta}{(\nu  - 1)^2}  (1 -   x_i)^2\\
 & = & \left((1 - \lambda)^2 - \alpha - \frac{\beta}{(\nu  - 1)^2}\right)   x_i^2 + \left(2 \lambda (1 - \lambda) + \alpha - \frac{\beta}{\nu  - 1} +  \frac{2 \beta}{(\nu  - 1)^2}\right)   x_i \\
 &  & \quad + \lambda^2 + \frac{\beta}{\nu  - 1} - \frac{\beta}{(\nu  - 1)^2} \, .\end{array}
 \end{equation}
We claim the last expression of~\eqref{crux} equals $(1 - \lambda)   x_i + \lambda = u_i$, which follows by equating the coefficients of $x_i^2$, $x_i$, and $1$, in above expression and re-arranging all terms with $\lambda$ to the right-hand side: 
 \begin{equation}\label{crux2}\left\{ .\begin{array}{rcccl}
 \alpha + \frac{\beta}{(\nu  - 1)^2} & = & (1 - \lambda)^2 & = & \frac{(\nu - \rho)^2}{(\nu - 1)^2}\\
 \alpha + \frac{\beta (3-\nu)}{(\nu  - 1)^2} & = &(1 - \lambda) (1- 2 \lambda) & = & \frac{(\nu - \rho)(\nu - 2 \rho + 1)}{(\nu - 1)^2} \\
 0 +  \frac{\beta (\nu-2)}{(\nu  - 1)^2} &=& \lambda (1-\lambda) & = & \frac{(\nu - \rho)(\rho - 1)}{(\nu  - 1)^2}
\end{array}
\right.\, . \end{equation}
Observe that the system \eqref{crux2} has a unique solution given by \eqref{def_alpha} and \eqref{def_beta} since subtracting the second equation from the first one yields the third equation. Therefore, we obtain that $\textrm{diag}(  U) =   u$. We clearly have $  X \geq 0$, $  R \geq 0$, and $  X -   R\T  =  (  x -   u)   x\T  \leq 0$. Finally, we focus on $  X -   R -   R\T  +   U \geq 0$ since each of $R - U \leq 0$ and $U \geq 0$ is implied by these constraints. If $  x_i = 0$, then $  U_{ii} - 2   R_{ii} +   X_{ii} = 0 \geq 0$. On the other hand, if $  x_i > 0$, we have
\[
  U_{ii} - 2   R_{ii} +   X_{ii} =    u_i^2 + U^1_{ii} + U^2_{ii} - 2   x_i   u_i +   x_i^2 = (  u_i -   x_i)^2 + U^1_{ii} + U^2_{ii} \geq 0,
\]
where we used $U^1 \succeq 0$ and $U^2 \succeq 0$. 
Similarly, $  U_{ij} -   R_{ij} -   R_{ji} +   X_{ij} = 0 \geq 0$ whenever $1 \leq i < j \leq n$ and $  x_i   x_j = 0$. On the other hand, if $1 \leq i < j \leq n$ and $  x_i   x_j > 0$, we obtain
\begin{eqnarray*}
  U_{ij} -   R_{ij} -   R_{ji} +   X_{ij} & = &   u_i   u_j - \alpha   x_i   x_j - \beta a_i a_j -   x_i   u_j -   x_j   u_i +   x_i   x_j \\
 & = & (  u_i -   x_i)(  u_j -   x_j) - \alpha   x_i   x_j - \beta a_i a_j \\
 & = & \lambda^2 (1 -   x_i) (1 -   x_j) - \alpha   x_i   x_j - {\textstyle\frac{\beta}{(\nu  - 1)^2}}\, (1 -   x_i) (1 -   x_j) \\
 & = & \left(\lambda^2 - {\textstyle\frac{\beta}{(\nu  - 1)^2}}\right) (1 -   x_i -   x_j) + \left(\lambda^2 - {\textstyle\frac{\beta}{(\nu  - 1)^2}} - \alpha\right)   x_i   x_j\\
 & = & {\textstyle\frac{(\rho - 1)(\rho - 2)}{(\nu  - 1)(\nu  - 2)}}\left(1 -   x_i -   x_j\right)  
 +  {\textstyle\frac{   2 \rho - 1 - \nu  }{\nu  - 1}} \,   x_i   x_j
 \, ,
\end{eqnarray*}
where we used \eqref{def_lambda}, \eqref{def_alpha}, and \eqref{def_beta} to derive the last equation. 
Since $\rho \geq 2$, $\nu \geq 3$, $1 -   x_i -   x_j\ge 0$ and $  x_i   x_j > 0$, it follows that $  U_{ij} -   R_{ij} -   R_{ji} +   X_{ij} \geq 0$ if $\nu  \leq 2 \rho - 1$, which establishes \eqref{incl1} and \eqref{incl2}. If, on the other hand, $\nu  > 2 \rho - 1$, then $  U_{ij} -   R_{ij} -   R_{ji} +   X_{ij} \geq 0$ by \eqref{suff_cond_set1}, giving rise to \eqref{incl3}. This completes the proof.
\end{proof}   

Before we proceed to the important consequences of the above result, let us motivate the construction in its proof, in particular the choice of $\lambda$ and the other constants.

\begin{observation}\label{obs}
  Let $\rho\in \{ 2,\ldots,n\}$ and let ${x}\in F$.  Assume that $ {u}_i = \tau {x}_i +b$ if $x_i>0$ with $0<\tau<1$, while $u_i=0$ if $x_i=0$. Furthermore, assume that $ {U}_{ij} = c {x}_i + c  {x}_j +d$ if $x_ix_j>0$ while $U_{ij}=0$ if $x_ix_j=0$ for $1 \leq i < j \leq n$. It is easy to verify that the choices of $u$ and $U$ in the proof of Theorem~\ref{general-construct} are in this form. Then, the best choice of $\tau$, $b$, $c$ and $d$ ensuring that $(x, u, X, U, R) = ( {x} ,  {u},  {x} {x}\T, {U}, {x} {u}\T)$ is R3($\rho$)-feasible, is the choice in the proof of Theorem~\ref{general-construct}.
    \end{observation}

    \begin{proof} Let $\rho\in \{ 2,\ldots,n\}$ and let ${x}\in F$. Again, abbreviate $\nu=\| {x}\|_0$.
   From $e\T (\tau {x} +b) = e\T u= \rho$, we derive $b = \frac{\rho-\tau}{\nu}\in (0,1)$ as $\rho>1>\tau$ and $\rho-\tau< \rho < \nu$. Furthermore, the constraints $x\leq u \leq e$ become
   $$x_i \le \min \left \{ {\frac{1-b}{\tau}, \frac{b}{1-\tau}}\right \}
   =  \min \left \{{\frac{ \nu-\rho+\tau}{ \nu \tau}, \frac{\rho-\tau}{ \nu(1-\tau)}}\right \}
   \quad\mbox{for all }i = 1,\ldots,n\, .$$
   Since 
  $g(\tau):=\frac{ \nu-\rho+\tau}{ \nu \tau}$  decreases and $h(\tau):=\frac{\rho-\tau}{ \nu(1-\tau)}$ increases with $\tau\in (0,1)$, the maximum of  $\min \left \{ g(\tau), h(\tau)\right \}$ is attained at $\tau^*$ satisfying $g(\tau^*)=h(\tau^*)$, and this value ensures that the formulation covers as many $ x\in F $ as possible.
  Hence the best choice of $\tau$ would be the solution $\tau^*$ of $g(\tau^*)=h(\tau^*)$, namely $\tau^*=\frac{ \nu-\rho}{ \nu-1}$,
    which is exactly our choice in the proof of Theorem~\ref{general-construct} with $\lambda = 1 - \tau^*= \frac{\rho-1}{\nu-\rho}$.   
    Since $U_{ii} = u_i$ and $Ue = \rho u$, we have for $x_i>0$
    $$\begin{aligned}
        \tau^*  x _i + b + \sum_{j\neq i:x_j>0}(c x _i + c  x _j +d) &=\rho(\tau^*  x _i + b)\quad\mbox{or} \\[0.3em]
        ( \nu -1)d + c + \frac{\rho - \tau^*}{ \nu} + ( \nu -1)c  x _i + (\tau^* - c) x _i &= \rho\tau^*  x _i + \frac{\rho(\rho - \tau^*)}{ \nu}, 
    \end{aligned}$$
    which implies, comparing coefficients of $x_i$ and 1, that 
    \[ \begin{cases}
   ( \nu -1)c + \tau^* - c = \rho \tau^*     \quad\mbox{and}\\[0.3em]
    ( \nu -1)d + c + \frac{\rho - \tau^*}{ \nu} = \frac{\rho(\rho - \tau^*)}{ \nu}     \, ,
    \end{cases}\]
   so that $c = \frac{(\rho-1)\tau^*}{ \nu-2}= \frac{(\rho-1)(\nu-\rho)}{(\nu-1)(\nu-2)}$ and $d = 
    \frac{\rho-1}{\nu(\nu-1)(\nu-2)}[(\nu-2)\rho-2\tau^*( \nu-1)] =\frac{(\rho-1)(\rho-2)}{ ( \nu-1)( \nu-2)}$,
   substituting $\tau^* = \frac{ \nu-\rho}{ \nu-1}$. This justifies our choice of $c$ and $d$ in the proof of Theorem~\ref{general-construct}. 
\end{proof}

 \begin{example}\label{exle} The condition~\eqref{suff_cond_set1} is sufficient but not necessary. For $n = 6$ and $\rho = 3$, the point 
 \[ 
 x = [0.6,0.2,0.05,0.05,0.05,0.05]\T \in F 
 \]
 violates~\eqref{suff_cond_set1} since 
    $$0.6 =\frac{(0.6)\, (0.2)}{1 - 0.6 - 0.2} = \max\limits_{1 \leq i < j \leq n: x_i x_j > 0} \frac{x_i x_j}{1 - x_i - x_j}>\frac{(\rho - 1)(\rho - 2)}{\left(\|x\|_0 - 2\right) \left(\|x\|_0 - 2 \rho + 1\right)} = 0.5\, ,$$
    while there exists $(u, U, R) \in \R^6 \times \cS^6 \times \R^{6 \times 6}$ such that $(x, u, X, U, R) = ( {x} ,  {u},  {x} {x}\T, {U}, {x} {u}\T)$ is (SDP-RLT(3))-feasible.
    One choice of $u$ and $U$ is $u = [0.8866,0.5512,0.3905,0.3906,0.3906,0.3905]\T$ and $$U = \begin{bmatrix}
    0.8866    &0.4674    &0.3264    &0.3265    &0.3265    &0.3264\\
    0.4674    &0.5512    &0.1588    &0.1588    &0.1588    &0.1588\\
    0.3264    &0.1588    &0.3905    &0.0986    &0.0986    &0.0986\\
    0.3265    &0.1588    &0.0986    &0.3906    &0.0986    &0.0986\\
    0.3265    &0.1588    &0.0986    &0.0986    &0.3906    &0.0986\\
    0.3264    &0.1588    &0.0986    &0.0986    &0.0986    &0.3905\\
    \end{bmatrix}\, .$$
Note that $u$ is not given by an affine function of $x$ in the sense of Observation~\ref{obs}.
\end{example}
    
Theorem~\ref{general-construct} reveals that an increasingly larger and nontrivial set of rank-one solutions is contained in the sets $\cF^{R3}_\rho$ as $\rho$ increases. Note that $G_\rho$ given by \eqref{suff_cond_set1} is a nonconvex set. Our next result gives further insight into this set by providing a piecewise convex inner approximation.

\begin{lemma} \label{inner_approx_Grho}
We have $G_2 = \emptyset$ and  
\begin{equation} \label{def_Hrho}
H_\rho := \bigcup\limits_{\nu = 2 \rho}^n \left\{x \in F: \| x \|_0 = \nu, \quad x_i + x_j \leq \delta_{\rho,\nu},~1 \leq i < j \leq n\right\} \subseteq G_\rho  \quad\mbox{if } \rho \in \left\{3,\ldots,\left \lfloor \textstyle{\frac{n}{2}} \right \rfloor \right\},  
\end{equation}
where $G_\rho$ is defined as in \eqref{suff_cond_set1} and 
\begin{equation} \label{def_delta}
\delta_{\rho,\nu} := 2 \left[\left(\tau_{\rho,\nu}^2 + \tau_{\rho,\nu}\right)^{1/2} - \tau_{\rho,\nu}\right],   
\end{equation}
with
\begin{equation} \label{def_tau}
\tau_{\rho,\nu} := \frac{(\rho - 1)(\rho - 2)}{\left(\nu - 2\right) \left(\nu - 2 \rho + 1\right)}.    
\end{equation}
Furthermore,
\begin{equation} \label{incl4}
\left\{(x, x x\T ) \in \cF^{R3}: x \in H_\rho \right\} \subseteq \cF^{R3}_\rho \quad \mbox{for all }\rho \in \left\{ 3, \ldots, \left \lfloor \textstyle{\frac{n}{2}} \right \rfloor\right\}\, .
\end{equation}
\end{lemma}
\begin{proof}
For $\rho = 2$, the upper bound in \eqref{suff_cond_set1} equals zero, which implies that $G_2 = \emptyset$. Let us fix $\rho \in \left\{3,\ldots,\left \lfloor \textstyle{\frac{n}{2}} \right \rfloor \right\}$ and let $x \in H_\rho$. Then, $x \in F$, $\| x \|_0 = \nu > 2 \rho - 1$, and it is easy to verify that 
\[
\max\limits_{1 \leq i < j \leq n: x_i x_j > 0} \frac{x_i x_j}{1 - x_i - x_j} \leq \frac{\delta_{\rho,\nu}^2}{4 \left( 1 - \delta_{\rho,\nu} \right)} = \tau_{\rho,\nu} \, ,
\]
where the last equality follows from \eqref{def_delta} and \eqref{def_tau}. Both inclusions \eqref{def_Hrho} and \eqref{incl4} now follow from Theorem~\ref{general-construct} by observing that $\| x \|_0 = \nu$.
\end{proof}

For fixed $\rho \in \left\{3,\ldots,\left \lfloor \textstyle{\frac{n}{2}} \right \rfloor \right\}$, it is worth noticing that $\tau_{\rho,\nu}$ given by \eqref{def_tau} is a decreasing function of $\nu$, which, in turn, implies that $\delta_{\rho,\nu}$ given by \eqref{def_delta} is a decreasing function of $\nu$. Therefore, the positive components of the elements of $H_\rho$ given by Lemma~\ref{inner_approx_Grho} tend to get closer to each other as $\nu$ increases. For instance, if $\rho = 3$, then $\delta_{\rho,\nu}$ equals $0.7321$, $0.5798$, and $0.4805$ for $\nu = 6$, $\nu = 7$, and $\nu = 8$, respectively. Note that the point $x$ of~Example~\ref{exle} satisfies $x\notin G_3$, readily certifying $x\notin H_3$ since $x_1+x_2=0.8>0.7321$. 

Theorem~\ref{general-construct} gives rise to several results about rank-one solutions of $\cF^{R3}_\rho$. Our next result gives a complete description of such solutions for $\rho = 2$. 

\begin{corollary} \label{rank-one-rho0-3-rho-2}
 We have $(  x,   x   x\T ) \in \cF^{R3}_2$ if and only if $x \in F_3$.
\end{corollary}
\begin{proof}
By Theorem~\ref{general-construct}, for any $  x \in F_3$, we have $(  x,   x   x\T ) \in \cF^{R3}_2$ by~\eqref{incl1}. The assertion follows from Corollary~\ref{ub-rankone-rho-2}. 
\end{proof}

For $\rho = 1$ and $\rho = 2$, it follows from Corollary~\ref{rank-one-rho-1} and Corollary~\ref{rank-one-rho0-3-rho-2} that $(  x,   x   x\T ) \in \cF^{R3}_\rho$ if and only if $x \in F_1$ and $x \in F_3$, respectively. On the other hand, for $\rho \geq 3$, Theorem~\ref{general-construct} gives rise to our next result, which reveals that such a nontrivial upper bound on $\|   x \|_0$ concerning rank-one solutions of $\cF^{R3}_\rho$ does not exist.

\begin{lemma} \label{no-upper-bound}
    Let $\rho\geq 3$. Then, for any $\nu \in \{\rho + 1,\ldots,n\}$, 
    there exists $  x \in F_{\nu}$ such that $(  x,   x   x\T ) \in \cF^{R3}_\rho$. 
\end{lemma}
\begin{proof}
Let $\rho \geq 3$ and $\nu \geq \rho + 1$. By Theorem~\ref{general-construct}, the assertion clearly holds for any $  x \in F_{\nu}$ such that $\|   x \|_0 = \nu \leq 2 \rho - 1$. Suppose that  $\nu > 2 \rho - 1$. By Lemma~\ref{inner_approx_Grho}, it suffices to construct an $  x \in F_{\nu}$ such that $\|   x \|_0 = \nu$ and $x \in H_\rho$, where $H_\rho$ is defined as in \eqref{def_Hrho}. Let $  x \in F_{\nu}$ be given by
\[
  x_i = \begin{cases}
\frac{1}{\nu}\,, & \textrm{if } i \in\{ 1,\ldots,\nu\}\, ,\\
0\, , & \textrm{otherwise.}
\end{cases}
\]
We therefore need to verify that
\[
\frac{2}{\nu} \leq 2 \left[\left(\tau_{\rho,\nu}^2 + \tau_{\rho,\nu}\right)^{1/2} - \tau_{\rho,\nu}\right]\, ,
\]
where $\tau_{\rho,\nu}$ is given by \eqref{def_tau}. Rearranging and simplifying the terms, the above inequality reduces to
\[
\frac 1{\nu(\nu - 2)} \leq \tau_{\rho,\nu}\, .
\]
By \eqref{def_tau}, this inequality holds if 
\[
\frac{\nu - 2 \rho + 1}{\nu} \leq (\rho - 1)(\rho - 2)\, .
\]
Since $\rho \geq 3$ (and thus $2\rho-1>0$), we even have
\[
\frac{\nu - 2 \rho + 1}{\nu} \leq 1 \leq (\rho - 1)(\rho - 2)\, ,
\]
which establishes the assertion.
\end{proof}

Following our earlier discussion about the positive components of the elements of the set $H_\rho$, we remark that all such components of the solution constructed in the proof of Lemma~\ref{no-upper-bound} are equal. Our next result establishes another useful property of the rank-one solutions of $\cF^{R3}_\rho$.

\begin{theorem} \label{nested-rank-one}
For each $\rho \in \{1,\ldots,n-1\}$, if $(x, xx\T) \in \cF^{R3}_\rho$, then $(x, xx\T) \in \cF^{R3}_{\rho+1}$.   
\end{theorem}
\begin{proof}
Let $\rho \in \{1,\ldots,n-1\}$ and let $(x, xx\T) \in \cF^{R3}_\rho$. Let us define $\nu = \| x \|_0$. If $\rho\in \left\{ 2,\ldots,\left \lfloor \textstyle{\frac{n+1}{2}}\right \rfloor - 1\right\}$ and $\nu \leq 2 (\rho + 1) - 1 = 2 \rho + 1$; or if $\rho \in \left\{  \left \lfloor \textstyle{\frac{n+1}{2}} \right \rfloor, \ldots,n\right\}$, then the assertion follows from Theorem~\ref{general-construct}. Therefore, let us assume that $\rho\in \left\{ 2,\ldots,\left \lfloor \textstyle{\frac{n+1}{2}}\right \rfloor - 1\right\}$ and $\nu > 2 \rho + 1$. For each $\rho \geq 3$, we remark that the set of rank-one solutions with this property is nonempty by Lemma~\ref{no-upper-bound}.

Since $(x, xx\T) \in \cF^{R3}_\rho$, there exists $(u,U,R) \in \R^n \times \cS^n \times \R^{n \times n}$ such that $( x,  u,  x  x\T ,  U,  R) \in \R^n \times \R^n \times \cS^n \times \cS^n \times \R^{n\times n}$ is (R3($\rho$))-feasible. Since $X = x x\T$, we have $R = x u\T$ and $U - u u\T \succeq 0$ by Schur complementation. We will construct $(u^\prime, U^\prime, R^\prime) \in \R^n \times \cS^n \times \R^{n \times n}$ such that $( x,  u^\prime,  x  x\T ,  U^\prime,  R^\prime) \in \R^n \times \R^n \times \cS^n \times \cS^n \times \R^{n\times n}$ is (R3($\rho + 1$))-feasible.

Let $u^\prime = u + s$, where $s \in \R^n$ is given by
    \[
    s_i = \begin{cases}
        \frac{1 - u_i}{\Vert u \Vert_0 - \rho}, & \text{if $u_i > 0$},\\
        0, & \text{otherwise.}
    \end{cases}
    \]
Since $0 \leq u \leq e$, we have $s \in \R^n_+$ since $\Vert u \Vert_0 \geq \nu > 2 \rho + 1$, which implies that 
\[
\mu := \Vert u \Vert_0 - \rho > \rho + 1 \geq 3. 
\]
Therefore, we obtain $0 \leq x \leq u \leq u^\prime \leq e$. Furthermore, $e\T s = 1$, which implies that  $e\T u^\prime = \rho + 1$. Since $X = x x\T$, we define $R^\prime = x (u^\prime)^T = R + x s\T$. Finally, we define
\[ 
U^\prime = u^\prime (u^\prime)^T + \frac{\mu - 2}{\mu} \left( U - uu^T \right) + \textrm{Diag}(s) - ss^T.
\]
By Schur complementation, 
\[
\begin{bmatrix} X & R^\prime \\ (R^\prime)\T & U^\prime \end{bmatrix} - \begin{bmatrix} x \\ u^\prime \end{bmatrix} \begin{bmatrix} x \\ u^\prime \end{bmatrix}\T = \begin{bmatrix} 0 & 0 \\ 0 & \frac{\mu - 2}{\mu} \left( U - uu^T \right) + \textrm{Diag}(s) - ss^T \end{bmatrix} \succeq 0,
\]
where we used Lemma~\ref{psd-condition}, $\mu > 3$, and $U - u u\T \succeq 0$. Therefore, 
the semidefiniteness constraint is satisfied. We clearly have $X e = x$, $R^\prime e = (\rho + 1) x$, $(R^\prime)\T e = u^\prime$, and $U^\prime e = (\rho + 1) u^\prime$. We next focus on the constraint $\textrm{diag}(U^\prime) = u^\prime$. If $u_i = 0$, then $u^\prime_i = u_i = U_{ii} = U^\prime_{ii} = 0$ since $s_i = 0$. If $u_i > 0$, then
\begin{eqnarray*}
U^\prime_{ii} & = & (u^\prime_i)^2 + \frac{\mu - 2}{\mu} \left( U_{ii} - u_i^2 \right) + s_i - s_i^2 \\
 & = & \left( u_i + s_i \right)^2 + \frac{\mu - 2}{\mu} \left( u_i - u_i^2 \right) + s_i - s_i^2 \\
 & = & \frac{2}{\mu} u_i^2  + \frac{\mu - 2}{\mu} u_i + s_i + 2 u_i s_i \\
 & = & \frac{1}{\mu} \left(2 u_i^2  + (\mu - 2) u_i + 1 - u_i + 2 u_i (1 - u_i) \right) \\
& = & \frac{(\mu - 1) u_i + 1}{\mu} \\
& = & u^\prime_i,
\end{eqnarray*}
where we used $\textrm{diag}(U) = u$ in the second line and the definition of $s$ in the fourth line. This establishes $\textrm{diag}(U^\prime) = u^\prime$.

Furthermore, we have $X \geq 0$, $R^\prime =R + x s\T \geq 0$, and $X - (R^\prime)\T = X - R\T - s x\T \leq 0$ since $X - R\T \leq 0$, $x \geq 0$, and $s \geq 0$. We next verify $X - (R^\prime)\T - R^\prime + U^\prime \geq 0$. Recall again that the remaining inequality constraints are implied. For the diagonal components, we have
\[
  U^\prime_{ii} - 2   R^\prime_{ii} +   X_{ii} =    u^\prime_i - 2   x_i   u^\prime_i +   x_i^2 \geq (u^\prime_i)^2 - 2   x_i   u^\prime_i +   x_i^2 = (  u_i -   x_i)^2 \geq 0, \quad i = 1,\ldots,n,
\]
where we used $\textrm{diag}(U^\prime) = u^\prime$ and $0 \leq x \leq u \leq e$. If $1 \leq i < j \leq n$, then
\begin{eqnarray*}
U^\prime_{ij} -   R^\prime_{ij} -   R^\prime_{ji} +   X_{ij} & = & u^\prime_i u^\prime_j + \frac{\mu - 2}{\mu} \left( U_{ij} - u_i u_j \right) - s_i s_j - x_i u^\prime_j - x_j u^\prime_i + x_i x_j \\
 & = & (u_i + s_i) (u_j + s_j) + \frac{\mu - 2}{\mu} \left( U_{ij} - u_i u_j \right) - s_i s_j \\
 & & \quad - x_i (u_j + s_j) - x_j (u_i + s_i) + x_i x_j \\
 & = & u_i s_j + s_i u_j - x_i u_j - x_i s_j - x_j u_i - x_j s_i + x_i x_j + \frac{2}{\mu} u_i u_j + \frac{\mu - 2}{\mu} U_{ij} \\
 & = & \frac{\mu - 2}{\mu} \left( U_{ij} - x_i u_j - x_j u_i + x_i x_j \right) + \frac{2}{\mu} \left( u_i u_j - x_i u_j - x_j u_i + x_i x_j \right) \\
 & & \quad + u_i s_j + s_i u_j - x_i s_j - x_j s_i \\
 & \geq & \frac{2}{\mu} \left( (u_i - x_i) (u_j - x_j) \right) + \left( s_j (u_i - x_i) + s_i (u_j - x_j) \right) \\
 & \geq & 0,
\end{eqnarray*}
where we used $\mu > 3$ and $  U_{ij} -   R_{ij} -   R_{ji} +   X_{ij} = U_{ij} - x_i u_j - x_j u_i + x_i x_j \geq 0$ to derive the first inequality, and $0 \leq x \leq u$ together with $s \geq 0$ to arrive at the final one. This completes the proof. 
\end{proof}

Theorem~\ref{nested-rank-one} establishes the nested behavior of the set of rank-one solutions of $\cF^{R3}_\rho$ with respect to $\rho$. We close this section with the following result about the tightness of the lower bound $\ell^{R3}_\rho(Q)$ arising from \eqref{R3}.

\begin{corollary} \label{quality_lb}
We have 
\begin{eqnarray} \label{quality_lb_rels}
\ell^{R3}_\rho(Q) & \leq & \ell_{2 \rho - 1}(Q)\, , \quad \textrm{if } \rho\in \left\{ 2, \ldots, \left \lfloor \textstyle{\frac{n+1}{2}} \right \rfloor\right\} \, , \textrm{ while} \label{quality_lb_rel1} \\
\ell^{R3}_\rho(Q) & \leq & \ell(Q)\, , \quad \quad \textrm{if }\rho \in \left\{  \left \lfloor \textstyle{\frac{n+1}{2}} \right \rfloor + 1, \ldots,n\right\}  \, .\label{quality_lb_rel2}   
\end{eqnarray}
\end{corollary}
\begin{proof}
The relations follow from \eqref{ell_rho_R3_alt} and from \eqref{incl1} and \eqref{incl2}, respectively.    
\end{proof}

Corollary~\ref{quality_lb} reveals that the lower bound $\ell^{R3}_\rho(Q)$ can be potentially quite weak especially for larger values of $\rho$.

\section{Concluding Remarks} \label{conc}

A Standard Quadratic optimization Problem with hard sparsity constraints can be exactly reformulated as a mixed-binary QP. Therefore, it is tempting to use tractable LP- or SDP-based relaxations, either in a straightforward/vanilla way or by suitable combinations as we did in Section~\ref{SDP-RLT-Relaxation}. The aim is to achieve tight rigorous bounds with a computational effort that scales well with the problem size. 

However, our analysis reveals that some caveats are in place when following this approach. In unfavorable circumstances (e.g., if the sparsity constraints are not stringent enough), the resulting bounds are quite weak.
We characterized the exactness of the bounds and studied the behavior of rank-one solutions to the relaxations. 

The findings of this article definitely call for more investigation, either in the direction of refined RLT models, or equally importantly, tighter conic-based relaxations which still offer some tractability. While these avenues are beyond the scope of the present work, they remain on our research agenda for the near future.

\bibliographystyle{abbrv}
\bibliography{references}

\end{document}